\documentclass[12pt]{amsart}
\newcommand{\mytitle}{Relative algebraic \texorpdfstring{$K$}{K}-theory by elementary means}
\newcommand{\myauthor}{Daniel R. Grayson}
\makeindex

\usepackage{amsmath,amssymb,amsfonts}
\usepackage{epsfig,cancel,cite}

\newlength{\xtrawidth}
\newcommand{\beq}{\begin{equation}}
\newcommand{\eeq}{\end{equation}}
\newcommand{\ba}{\begin{array}}
\newcommand{\ea}{\end{array}}
\newcommand{\bea}{\begin{eqnarray}}
\newcommand{\eea}{\end{eqnarray}}
\newcommand{\bean}{\begin{eqnarray*}}
\newcommand{\eean}{\end{eqnarray*}}

\newcommand{\comment}[1]{}

\newcommand{\Om}{\Omega}

\newcommand{\Gr}{\mathop{Gr}}

\newcommand{\gr}{\mathop{gr}}

\newcommand{\Ab}{{\mathcal Ab}}
\newcommand{\coker}{\mathop{{\rm coker}}}

\newcommand{\ZZ}{\mathbb{Z}}
\newcommand{\NN}{\mathbb{N}}
\newcommand{\cA}{{\mathcal A}}

\newcommand{\cC}{{\mathcal C}}
\newcommand{\cD}{{\mathcal D}}

\newcommand{\cM}{{\mathcal M}}
\newcommand{\cN}{{\mathcal N}}

\theoremstyle{plain}
  \newtheorem{theorem}            {Theorem}
  \newtheorem{corollary} [theorem]{Corollary}
  
  \newtheorem{lemma}     [theorem]{Lemma}

  \newtheorem{conjecture}[theorem]{Conjecture}

\theoremstyle{definition}
  \newtheorem{definition}[theorem]{Definition}
  \newtheorem{remark}    [theorem]{Remark}

\usepackage[all]{xy} 
\newdir{ >}{{}*!/-10pt/@{>}}   

\newcommand{\Tot}{\mathop{{\rm Tot}}}
\newcommand{\Cone}{\mathop{{\rm Cone}}}
\newcommand{\isom}{\cong}

\newcommand{\Top}{{\scriptscriptstyle\top}}
\newcommand{\Bot}{{\scriptscriptstyle\bot}}
\newcommand{\topbot}{\hbox{\hbox to 0pt{$\Bot$\hss}$\Top$}}
\usepackage{mathpartir}         
\usepackage{mathtools}          
\newcommand{\defeq}{\vcentcolon=}
\newcommand{\src}{{\rm src}}
\newcommand{\sw}{\tau}
\newcommand{\tar}{{\rm tar}}
\newcommand{\qi}{{\rm comp}}

\usepackage{hyperref} 
\hypersetup{pdftitle={\mytitle}}
\hypersetup{pdfauthor={\myauthor}}

\renewcommand{\autoref}{\ref}

\begin{document}
\title{\mytitle}
\author{\myauthor}
\email{drg@illinois.edu}
\urladdr{\href{http://dangrayson.com/}{http://dangrayson.com/}}
\def\emailaddrname{{\itshape Email}}
\def\urladdrname{{\itshape Home page}}
\date{}

\abstract

  In a previous paper I gave a presentation for the Quillen higher algebraic $K$-groups of an exact category in terms of {\em acyclic binary
    multicomplexes}.  In this paper I take that presentation as a definition of the higher $K$-groups, generalize it to the relative $K$-groups
  of an exact functor between exact categories, and produce the corresponding long exact sequence by elementary means, without homotopy theory.

\endabstract

\subjclass[2000]{19D99}
\keywords{algebraic K-theory, acyclic binary complexes, generators and relations, higher algebraic K-groups}
\maketitle

\numberwithin{theorem}{section}

\section*{Introduction}

  The paper \cite{grayson-MR2947948} provided the first presentation by generators and relations for the higher algebraic $K$-groups of Quillen,
  which are defined as homotopy groups of certain spaces constructed by combinatorial means from the algebraic situation under consideration.
  The generators and relations involve ingredients just a bit more complicated than the notion of chain complex from homological algebra.

  A fundamental aspect of algebraic $K$-theory is the construction and use of long exact sequences of algebraic $K$-groups, such as those
  provided by the localization theorem for abelian categories or the localization theorem for projective modules of Quillen.  The most general
  of these is the long exact sequence
  \[ \dots \to K_{n+1} \cN \to K_{n+1}[F] \to K_n \cM \to K_n \cN \to \dots \]
  for an exact functor $F : \cM \to \cN$, involving the {\em relative} $K$-groups of $F$, which are also defined as homotopy groups.

  This paper\footnote{~Permanent ID of this document: 5ed1e5cfdacba5c670e4dfab8c111974; Date: 2013.12.2.} presents a conjectural presentation for the relative $K$-groups, adopts it as a definition, and then uses it directly to
  construct the long exact sequence by elementary means, without homotopy theory.  

  Hence the long exact sequence of relative $K$-theory in algebraic $K$-theory is intermediate in complexity between the long exact sequence in
  homotopy theory and the long exact sequence in homological algebra.  The result can also be viewed as an imaginary reconstruction of history,
  providing a step toward an approach to the invention of higher algebraic $K$-theory without homotopy theory and could have been done in the
  1960's.

  Perhaps the techniques developed here for computing with the presentation provided here can serve as tools for the construction of useful new
  long exact sequences of $K$-groups.

  I intend to settle the conjecture \ref{conj1} in a future paper.

\section*{Acknowledgments}

  I thank the University of Bielefeld (where much of this work was done) for its hospitality, the National Science Foundation for support under
  grant NSF DMS 10-02171, and the Oswald Veblen Fund and the Bell Companies Fellowship for supporting my stay at the Institute for Advanced
  Study in 2013-2014, where the paper was written.


\section{Basic constructions}

  In \cite[Definition 1.4]{grayson-MR2947948} we defined what it means for an exact category $\cN$ to {\em support long exact sequences}, what a
  {\em long exact sequence} of $\cN$ is \cite[Definition 1.1]{grayson-MR2947948}, and what a {\em quasi-isomorphism} of $\cN$ is
  \cite[Definition 2.6]{grayson-MR2947948}.  Any exact category can be converted to one that supports long exact sequences by adding objects
  representing images of idempotent maps, and the conversion changes only $K_0 \cN$.  The advantage of working in such a category is that all
  the usual statements from homological algebra about exactness of bounded chain complexes hold \cite[text after Definition
    1.4]{grayson-MR2947948}.  Without further comment, we assume that the exact categories mentioned here all support long exact sequences.  In
  \cite{grayson-MR2947948} are various results that state that the constructions we use for making new exact categories from old ones (for
  example, categories of chain complexes or binary chain complexes) preserve the property of supporting long exact sequences; similar statements
  hold for the slightly more advanced constructions of this paper.

  We let $\Gr\cN$ denote the exact category of bounded $\ZZ$-graded objects $N$ of $\cN$.  As in \cite[Definition 2.1, Remark
    2.2]{grayson-MR2947948}, for an exact category $\cN$, we let $C \cN$ denote the exact category of bounded chain complexes $(N,d)$ of objects
  of $\cN$, and as in \cite[Definition 3.1]{grayson-MR2947948} we let $B \cN$ denote the exact category of bounded binary chain complexes
  $(N,d,d')$ of objects of $\cN$; here $d$ is the {\em top} differential, and $d'$ is the {\em bottom} differential.  We also introduce
  $C^2\cN$ as an abbreviation for $C\cN^2$, the category of pairs $((N,d),(N',d'))$ of pairs of chain complexes of objects of $\cN$, because
  its use alternates so often with the use of $B\cN$.

  We introduce the following exact functors.

  \begin{mathparpagebreakable}
    \xymatrixrowsep {5pt}
    \xymatrix{C \cN \ar@{->}[r]^\gr & \Gr \cN \\ (N,d) \ar@{|->}[r] & N } \and
    \xymatrix{B \cN \ar@{->}[r]^\gr & \Gr \cN \\ (N,d,d') \ar@{|->}[r] & N } \\
    \xymatrix{C \cN \ar@{->}[r]^\Delta & C^2 \cN \\ (N,d) \ar@{|->}[r] & ((N,d),(N,d)) } \and
    \xymatrix{C \cN \ar@{->}[r]^\Delta & B \cN \\ (N,d) \ar@{|->}[r] & (N,d,d) } \\
    \xymatrix{B \cN \ar[r]^\Top & C \cN \\ (N,d,d') \ar@{|->}[r] & (N,d) } \and
    \xymatrix{B \cN \ar[r]^\Bot & C \cN \\ (N,d,d') \ar@{|->}[r] & (N,d') } 
    \xymatrix{B \cN \ar@{->}[r]^\topbot & C^2 \cN \\ (N,d,d') \ar@{|->}[r] & ((N,d),(N,d')) } \\
    \xymatrix{C^2 \cN \ar@{->}[r]^\sw & C^2 \cN \\ ((N,d),(N',d')) \ar@{|->}[r] & ((N',d'),(N,d)) } \and
    \xymatrix{B \cN \ar@{->}[r]^\sw & B \cN \\ (N,d,d') \ar@{|->}[r] & (N,d',d) }
  \end{mathparpagebreakable}

An object in the image of a functor $\Delta$ is called {\em diagonal}.

When it is likely to cause little confusion, we identify an object of $B\cN$ with its image under $\topbot$.  But beware: the faithful functor
$\topbot : B\cN \to C^2\cN$ is not full, because a map $(M,c,c') \to (N,d,d')$ in $B\cN$ is a single map $M \to N$ in $\Gr \cN$ that commutes
with both differentials.

In any of our exact categories, we let $i$ denote the subcategory of isomorphisms, and in any of our exact categories of chain complexes (or
binary chain complexes, etc.), we let $q$ denote the subcatgory of quasi-isomorphisms.  As explained in the second paragraph of \cite[Section
  4]{grayson-MR2947948}, the map $K \cN \to Kq C \cN$ arising from the embedding $\cN \hookrightarrow C \cN$ is a homotopy equivalence of
spectra.

Now we pass to relative $K$-theory for an exact functor $F : \cM \to \cN$ between exact categories, and we introduce several definitions.

\begin{definition}\label{CF}
  Let $C[F]$ be the exact category of triples $X = (M,N,u)$, where $M \in C\cM$, $N \in C\cN$, and $u : FM \xrightarrow\sim N$ is a
  quasi-isomorphism of $C \cN$.  A map $X \to X'$ is a pair $(f,g)$ consisting of maps $f : M \to M'$ and $g : N \to N'$ such that $u' \circ Ff
  = g \circ u$.  We call $u$ the {\em comparison} map. We introduce the notations $X_\src \defeq M$, $X_\tar \defeq N$, and $X_\qi \defeq u$.
\end{definition}

\begin{definition}\label{BF}
  Let $B[F]$ be the exact category of triples $X = (M,N,u)$, where $M \in C^2\cM$, $N \in B\cN$, and $u : FM \xrightarrow\sim \topbot N$ is a
  quasi-isomorphism of $C^2 \cN$.  We call $u$ the {\em comparison} map. A map $(M,N,u) \to (M',N',u')$ is a pair $(f,g)$ consisting of maps $f
  : M \to M'$ in $C^2\cM$ and $g : N \to N'$ in $B\cN$ such that $u' \circ Ff = \topbot g \circ u$.  We introduce the notations $X_\src \defeq
  M$, $X_\tar \defeq N$, $X_\qi \defeq u$ and the following exact functors.
  \begin{mathparpagebreakable}
    \xymatrixrowsep {7pt}
      \xymatrix{C[F] \ar@{->}[r]^\Delta & B[F] \\ (M,N,u) \ar@{|->}[r] & (\Delta M,\Delta N,\Delta u) } \and
      \xymatrix{B[F] \ar@{->}[r]^\Top & C[F] \\ (M,N,u) \ar@{|->}[r] & (\Top M,\Top N,\Top u) } \and
      \xymatrix{B[F] \ar@{->}[r]^\Bot & C[F] \\ (M,N,u) \ar@{|->}[r] & (\Bot M,\Bot N,\Bot u) } \and
      \xymatrix{B[F] \ar@{->}[r]^\topbot & C[F^2] \\ (M,N,u) \ar@{|->}[r] & (M,\topbot N,u) } \and
      \xymatrix{B[F] \ar@{->}[r]^\sw & B[F] \\ (M,N,u) \ar@{|->}[r] & (\sw M,\sw N,\sw u) }
  \end{mathparpagebreakable}
  Here $F^2 \defeq F \times F : \cM^2 \to \cN^2$, and we freely identify such categories as $C(\cM^2)$ and $(C\cM)^2$.
\end{definition}

For $X \in C[F]$ or $X \in B[F]$ we say that $X$ is {\em acyclic} if $X_\src$ and $X_\tar$ are acyclic.

\begin{definition}\label{pCF}
  Let $p \subseteq C[F]$ be the subcategory whose arrows $f : X \to X'$ are those where $f_\src : X_\src \xrightarrow \sim X'_\src$ is a
  quasi-isomorphism of $C\cM$ and $f_\tar : X_\tar \xrightarrow \isom X'_\tar$ is an isomorphism of $C\cN$.  Identifying $f$ with the pair
  $(f_\src,f_\tar) \in q \times i$ suggests the mnemonic notation $(q,i)$ for $p$.
\end{definition}

Observe that $p C[F]$ is an exact category with weak equivalences, in the sense of \cite[Definition A.1]{grayson-MR2947948}.  The category $p$
does not satisfy the cylinder axiom of Waldhausen \cite[1.6]{waldhausen-MR802796}, because the rear projection of a mapping cylinder in $C\cN$
is not an isomorphism.

\begin{definition}\label{pBF}
  Let $p \defeq (q,i) \subseteq B[F]$ be the subcategory whose arrows $f : X \to X'$ are those where $f_\src : X_\src \xrightarrow \sim X'_\src$ is a
  quasi-isomorphism of $C^2\cM$ and $f_\tar : X_\tar \xrightarrow \isom X'_\tar$ is an isomorphism of $B\cN$.
\end{definition}

Observe that $p B[F]$ is an exact category with weak equivalences and that $\Delta : pC[F] \to pB[F]$ is an exact functor.  The category $p$
does not satisfy the cylinder axiom of Waldhausen.

A helpful mnemonic may be to view $p B[F]$ as a sort of homotopy pullback of the diagram
\[\xymatrix{
                  & iB\cN \ar @{->} [d]^\topbot \\
 qC^2\cM \ar[r]^F  & qC^2\cN 
}
\]

\begin{definition}\label{omF}
  Let $\Om[F]$ denote the pair $ [ pB[F] \xleftarrow {\Delta[F]} pC[F] ]$ of exact categories with weak equivalences, where {\em pair} has
  the meaning adopted in \cite[Definition 4.1]{grayson-MR2947948} and in the text after \cite[Definition A.3]{grayson-MR2947948}.  (We point the
  arrow leftward to emphasize the role of its target.)
\end{definition}

A helpful mnemonic may be to view $\Omega[F]$ as a sort of homotopy pullback of the diagram
\[\xymatrix{
                  & [ iB\cN \xleftarrow \Delta iC\cN ] \ar @{->} [d]^{(\topbot,1)} \\
 [ qC^2\cM \xleftarrow \Delta qC\cM ] \ar[r]^F  & [ qC^2\cN \xleftarrow \Delta qC\cN ]
}
\]

\begin{conjecture}\label{conj1}
  Passing to $K$-theory gives the following homotopy pullback square.
  \[\xymatrix{
   K \Omega[F] \ar[r] \ar[d]                                  & K [ iB\cN \xleftarrow \Delta iC\cN ] \sim * \ar @{->} [d]^{(\topbot,1)} \\
   K\cM \sim K [ qC^2\cM \xleftarrow \Delta qC\cM ] \ar[r]^{KF}  & K [ qC^2\cN \xleftarrow \Delta qC\cN ] \sim K\cN
  }
  \]
\end{conjecture}

The upper right hand corner in the square above was shown to be contractible in \cite[proof of Theorem 4.3]{grayson-MR2947948}.  Commutativity
of the square up to homotopy is provided by the maps $u$, which are in $q$.

Recalling the definition \cite[Definition A.4]{grayson-MR2947948}, observe that the Grothendieck group $K_0 \Om[F]$ is the quotient of $K_0
B[F]$ that equates objects in the image of $\Delta[F]$ to $0$ and equates the source and the target of maps in $p$.  Alternatively, and better
for our purposes, is to take that description as the definition of $K_0 \Om[F]$.

The following theorem is our main result; the proof will be given later.

\begin{theorem}\label{exseqK00000}
  Given an exact functor $F : \cM \to \cN$ between exact categories that support long exact sequences, the sequence
  \[ \xymatrix {
    &
  K_0\Om[0 \to \cM] \ar[r] &
  K_0\Om[0 \to \cN] \ar[r] &
  K_0\Om[\cM \xrightarrow F \cN] \ar `r[d] `[l] `[llld] `[dll]  [dll] \\
    &
  K_0\Om[\cM \to 0] \ar[r] &
  K_0\Om[\cN \to 0] &
  } \] 
  of abelian groups induced by the maps of pairs
  \[
    [0 \to \cM] \xrightarrow {(1,F)}
    [0 \to \cN] \xrightarrow {(0,1)}
    [\cM \xrightarrow F \cN] \xrightarrow{(1,0)}
    [\cM \to 0] \xrightarrow {(F,1)}
    [\cN \to 0] 
  \]
  is exact.
\end{theorem}

Recall the notations $B^q\cN$ and $C^q\cN$ for the subcategories of acyclic objects in $B\cN$ and $C\cN$ respectively, introduced in the last
paragraphs of \cite[Sections 2 and 3]{grayson-MR2947948}.  Recall also the notation $\Om \cN \defeq [iB^q\cN \xleftarrow \Delta iC^q\cN]$,
which was introduced in \cite[Definition 4.2]{grayson-MR2947948}.

\begin{lemma}\label{relKspecial}
  Let $\cN$ be an exact category.  Then $K_0\Om[\cN \to 0] \isom K_0 \cN$ and $K_0\Om[0 \to \cN] \isom K_1 \cN$.
\end{lemma}

\begin{proof}
  Observe that $K_0 \Om[\cN \to 0] \isom K_0 [ qC^2\cN \xleftarrow \Delta qC\cN ] \isom \coker ( K_0 qC^2\cN \xleftarrow \Delta K_0 qC\cN ) \isom
  K_0 qC\cN \isom K_0 \cN$.  

  An object $(0,N,u) \in B[0 \to \cN]$ is essentially the same thing as an object $N \in B^q\cN$, because the unique arrow $u : 0 \to \topbot N$
  is a quasi-isomorphism if and only if $N \in B^q\cN$.  Similarly for an object of $C[0 \to \cN]$.  Observe then that $K_0\Om[0 \to \cN] \isom
  K_0[i B^q\cN \xleftarrow \Delta i C^q\cN] \isom K_0 \Om \cN \isom K_1 \cN$; the final isomorphism of the chain is provided by \cite[Corollary
    7.2]{grayson-MR2947948}.
\end{proof}

For our current purpose, we may adopt either $K_0\Om[0 \to \cN]$ or $K_0\Om\cN$ as the definition of $ K_1 \cN $, and we identify a class
$[(0,N,0)] \in K_0\Om[0 \to \cN]$ with the corresponding class $[N] \in K_1 \cN$.

\begin{corollary}\label{exseqK10}
  Given an exact functor $F : \cM \to \cN$ between exact categories that support long exact sequences, there is an exact sequence 
  \[ \xymatrix {
    K_1 \cM \ar[r] &
    K_1 \cN \ar[r] &
    K_0\Om[F] \ar[r] & 
    K_0 \cM \ar[r] &
    K_0 \cN & 
  } \]
\end{corollary}

\begin{remark}
  An explicit computation shows that, in general, $K_0\Om[F]$ is not isomorphic to $\coker (K_1 \cM \to K_1 \cN) \oplus \ker (K_0 \cM \to K_0
  \cN)$.  It compares the exact sequence above to Quillen's localization sequence for the embedding of a Dedekind domain with nontrivial ideal
  class group into its field of fractions.
\end{remark}

\section{The long exact sequence}

To extend the long exact sequence of \ref{exseqK10} to the left, we introduce the notion of {\em split cubes} of exact categories.

\begin{definition}
  A pair $[f]$ in a category $\cC$ is {\em split} if there is a map $g$ (called the {\em splitting}) with $g \circ f = 1$.  A {\em split
    $0$-cube} of $\cC$ is an object of $\cC$.  For $n \in \NN$ a {\em split $n+1$-cube} is a split pair $[C' \to C]$ in the category of split
  $n$-cubes of $\cC$.  A {\em split $n$-cube} of exact categories is a split $n$-cube in the category of exact categories.
\end{definition}

Recall from \cite[Section 7, first paragraph]{grayson-MR2947948} that the functor $\Om : \cN \mapsto \Om \cN$ can be iterated $n$ times,
yielding a functor $\Om^n : \cN \mapsto \Om^n \cN$, where $\Om^n \cN$ is regarded here as a split $n$-cube of exact categories.  It is split
because the functor $\Delta$ is split by the functor $\Top$.

Since the functor $\Top$ is natural, the functor $\Om$ can be viewed as a split pair of functors from exact categories to exact categories, and
thus $\Om^n$ can be viewed as a split $n$-cube of such functors.  In particular, if $[\cM \to \cN]$ is a pair of exact categories, then $\Om^n [
  \cM \xrightarrow F \cN] \defeq [ \Om^n\cM \xrightarrow F \Om^n\cN]$ is a split $n$-cube of pairs of exact categories -- that is a stronger
condition than being a pair of split $n$-cubes, since the splittings are not part of the data incorporated into a cube.

Observe that if $G : \cC \to \Ab$ is a functor from a category $\cC$ to the category $\Ab$ of abelian groups, then we can extend it to a functor
on split $n$-cubes in $\cC$, by induction on $n$, by defining $G[C' \to C] \defeq \coker (G C' \to G C) $.  Since a direct summand of an exact
sequence is exact, an exact sequence $G' \to G \to G''$ of functors $\cC \to \Ab$ remains exact when extended to a sequence of functors on split
$n$-cubes of $\cC$.

If $C$ is a split $n$-cube of exact categories and $m \in \NN$, then the discussion above gives a meaning to $K_m C$, as does the notion of
multi-relative $K$-theory of a cube of exact categories, discussed in \cite[Appendix]{grayson-MR2947948}.  The two meanings agree, because the
long exact sequence of relative $K$-theory for a split pair $[F : \cM \to \cN]$ splits into a collection of split short exact sequences $ 0 \to
K_m \cM \to K_m \cN \to K_m [F] \to 0$, demonstrating that the relative $K$-group $K_m [F]$ is isomorphic to $\coker ( K_m \cM \to K_m \cN )$.
Since the meanings agree, we may allow the ambiguity in notation, as in the statement of the following lemma.  For those readers interested only
in the elementary content of this paper, the second meaning can be ignored.

\begin{lemma}\label{Kmn}
  Given an exact category $\cN$ with $m \in \NN$ and $n \in \NN$, there is a natural isomorphism $K_m \Om^n \cN \isom K_{m+n}\cN$.
\end{lemma}

\begin{proof}
  We apply \cite[Corollary 7.1]{grayson-MR2947948} and use the notation used there: 
  $K_m \Om^n \cN = \pi_m K \Om^n \cN \isom \pi_m V^0 \Om^n K \cN \isom \pi_m \Om^n K \cN \isom K_{m+n}\cN$.
  Here $K \Om^n \cN$ is the multi-relative $K$-theory spectrum of the split $n$-cube $\Om^n \cN$ of exact categories,
  and $V^0$ denotes the functor that gives the $0$-th stage of the Postnikov filtration of a spectrum.
\end{proof}

Readers interested only in the elementary content of this paper may take $K_n \cN \defeq K_0 \Om^n \cN$ as a definition.  In that case, the
proof of the lemma would amount to the following chain of isomorphisms: $K_m \Om^n \cN \defeq (K_0 \Om^m) (\Om^n \cN) \isom K_0
(\Om^m \Om^n \cN) \isom K_0 (\Om^{m+n} \cN) \isom K_{m+n}\cN$.

\begin{corollary}
  Given an exact functor $F : \cM \to \cN$ between exact categories that support long exact sequences and $n \in \NN$, there is a long exact sequence 
  \[ 
  \xymatrix {
    \dots \ar[r] &
    K_{n+1} \cM \ar[r] &
    K_{n+1} \cN \ar[r] &
    (K_0 \Om)(\Om^n [ \cM \xrightarrow F \cN]) \ar `r[d] `[l] `[llld] `[dll] [dll] \\ 
    &
    K_n \cM \ar[r] &
    K_n \cN \ar[r] &
    \quad \quad \dots \quad \quad \ar `r[d] `[l] `[llld] `[dll] [dll] \\ &
    K_0 \cM \ar[r] &
    K_0 \cN &
  } 
  \]
  Here $K_0 \Om$ denotes the functor on pairs of exact categories discussed above, extended to split $n$-cubes of such pairs.
\end{corollary}

\begin{remark}
  Settling the conjecture \ref{conj1} would imply that $(K_0 \Om)(\Om^n [F]) \isom K_{n+1}[F]$ for $n \in \NN$.  The corollary doesn't
  address generators and relations for $K_0[F]$, but good ones are already provided by the isomorphism $K_0[F] \isom \coker ( K_0(F) : K_0\cM
  \to K_0\cN ) $.
\end{remark}

\begin{proof}
  We interpret the groups in the statement of \ref{exseqK10} as functors from the category of pairs $[F]$ of exact categories to the
  category of abelian groups, so we can extend them to functors on split $n$-cubes of pairs of exact categories.  The sequence remains exact
  when applied to $\Om^n [ \cM \xrightarrow F \cN]$, yielding the following exact sequence.
  \[ \xymatrix {
    K_1 \Om^n \cM \ar[r] &
    K_1 \Om^n \cN \ar[r] &
    (K_0\Om) (\Om^n [\cM \xrightarrow F \cN]) \ar[r] & 
    K_0 \Om^n \cM \ar[r] &
    K_0 \Om^n \cN & 
    } \]
  By \ref{Kmn} we can rewrite it as follows.
  \[ \xymatrix {
    K_{n+1} \cM \ar[r] &
    K_{n+1} \cN \ar[r] &
    (K_0\Om) (\Om^n [\cM \xrightarrow F \cN]) \ar[r] &
    K_n \cM \ar[r] &
    K_n \cN & 
    } \]
  Splicing these 5-term exact sequences together yields the desired result.
\end{proof}

\section{Proof of exactness, part 0}

\begin{lemma}
  Let $\cM$ be an exact category, and let $1_\cM : \cM \to \cM$ be the identity functor.  Then $K_0 \Om[1_\cM] = 0$.
\end{lemma}

\begin{proof}
  Consider a generator $\gamma = [(M,N,u)]$ of $K_0 \Om[1_\cM]$ with $(M,N,u) \in B[1_\cM]$.  The map $(u,1) : (M,N,u) \to (\topbot N, N, 1)$
  is in $p$, so $\gamma = [(\topbot N, N, 1)]$.  The successive subquotients $(\topbot N_i[-i], N_i[-i], 1)$ of the degree-wise filtration of
  $(\topbot N, N, 1)$ are in $B[1_\cM]$, because an identity map is a quasi-isomorphism, so we can write $\gamma = \sum_i [(\topbot N_i[-i],
    N_i[-i], 1)]$.  The object $(\topbot N_i[-i], N_i[-i], 1)$ is concentrated in degree $i$, so its differentials vanish, rendering its top and
  bottom differentials equal, and thus it is in the image of $\Delta$.  Hence each term in the sum vanishes.
\end{proof}

\begin{corollary}\label{seqiscomplex}
  The sequence of groups in \ref{exseqK00000} is a chain complex.  
\end{corollary}

\begin{proof}
  Examine the following commutative diagram, in which the lemma has been applied at three spots.
  \[\xymatrix{
    K_0\Om[0\to\cM] \ar[r]^{(1,F)}\ar[d]^{(0,1)} & K_0\Om[0\to\cN] \ar[r]^{(1,0)}\ar[d]^{(0,1)} & K_0\Om[0\to0] = 0\ar[d]^{(0,1)} \\
    0=K_0\Om[\cM\xrightarrow1\cM] \ar[r]^{(1,F)} & K_0\Om[\cM\xrightarrow F\cN] \ar[r]^{(1,0)}\ar[d]^{(F,1)} & K_0\Om[\cM\to0]\ar[d]^{(F,1)} \\
    &  0=K_0\Om[\cN\xrightarrow 1\cN] \ar[r]^{(1,0)} & K_0\Om[\cN\to0] 
    }\]
\end{proof}

\section{Facts about the Grothendieck group, part 1}

\begin{lemma}\label{K0lemma}
  Let $\cM$ be an exact category, and consider objects $M,M' \in \cM$.  Their classes $[M]$ and $[M']$ in $K_0\cM$ are equal if and only if
  there exist objects $V^0, V^1, V^2$ of $\cM$ and exact sequences of the form $E : 0 \to M \oplus V^0 \to V^1 \to V^2 \to 0$ and $E'
  : 0 \to M' \oplus V^0 \to V^1 \to V^2 \to 0$.  (We assume no relationship between the maps of $E$ and the maps of $E'$.)
\end{lemma}

\begin{remark}
  Clayton Sherman has notified me that this lemma occurs in a 1974 preprint of Steven Gersten, entitled {\em Higher K-Theory I : Products},
  where it is attributed to Richard Swan.
\end{remark}

\begin{proof}

  Use the notation $M \sim M'$ to indicate the condition that the three objects and two short exact sequences exist.  If $M \sim M'$, then
  $[M]=[M']$ follows from additivity and the existence of $E$ and $E'$.  

  We introduce some terminology.  A commutative (additive) monoid is {\em cancellative} if in it $x+z=y+z \Rightarrow x=y$ holds.  An additive
  equivalence relation is called {\em cancellative} if the corresponding quotient commutative monoid is.

  The relation $\sim$ is a cancellative additive equivalence relation on the commutative monoid consisting of the isomorphism classes of $\cM$,
  where direct sum gives the addition operation.  For example, additivity follows by forming direct sums of the exact sequences involved.  To
  show $M \sim M$ (reflexivity), one takes $E$ and $E'$ to be $0 \to M \xrightarrow 1 M \to 0 \to 0$; compatibility with isomorphism is similar.
  Symmetry follows from interchanging $E$ and $E'$.  To show $M \oplus W \sim M' \oplus W \Rightarrow M \sim M'$ (cancellation), one replaces
  $V^0$ by $V^0 \oplus W$.  To show $M \sim M' \sim M'' \Rightarrow M \sim M''$ (transitivity), one uses additivity to deduce that
  $M \oplus M' \sim M' \oplus M''$ and then cancels $M'$.
  
  Now let $H$ denote the commutative monoid of equivalence classes of objects of $\cM$ modulo $\sim$, with the equivalence class of $M$ denoted
  by $\langle M \rangle$.  Since it is cancellative, the universal map from $H$ to a group $G$ is injective.  To see that, construct $G$ in the
  standard way as a quotient set of $H \times H$, where $(h,h') \sim (k,k')$ if and only if $h+k'=h'+k$.  We may thus identify $h$ with the
  equivalence class of $(h,0)$.

  There is a well defined homomorphism $G \to K_0\cM$ defined by $(\langle M \rangle, \langle M' \rangle) \mapsto [M] - [M']$.  

  There is also a well defined homomorphism $K_0\cM \to G$ defined on generators by $[M] \mapsto \langle M \rangle$, because the relation is
  additive over short exact sequences.  To see that, consider an exact sequence $0 \to M' \xrightarrow i M \xrightarrow p M'' \to 0$, and use
  the following pair of exact sequences to show $M \sim M' \oplus M''$.
  \[
  \xymatrixrowsep {7pt}
  \xymatrix{ 
    0 \ar[r] &
    M \ar[r]^-{\begin{psmallmatrix}1\\0\end{psmallmatrix}} &
    M \oplus M'' \ar[r]^-{\begin{psmallmatrix}0&1\end{psmallmatrix}} &
    M'' \ar[r] &
    0 \\
    0 \ar[r] &
    M' \oplus M'' \ar[r]^-{\begin{psmallmatrix}i&0\\0&1\end{psmallmatrix}} &
    M \oplus M'' \ar[r]^-{\begin{psmallmatrix}p&0\end{psmallmatrix}} &
    M'' \ar[r] &
    0 
  }
  \] 
  The two homomorphisms are inverse to each other, as can be
  verified on generators, so they are isomorphisms.  Hence $\langle M \rangle = \langle M' \rangle$ if and only if $[M]=[M']$, yielding the
  result desired.
\end{proof}

\begin{corollary}\label{K0lemmacoro}
  The classes $[M]$ and $[M']$ in $K_0\cM$ are equal if and only if there exist three objects $V^0, V^1, V^2$ of $\cM$ and two exact sequences
  of the form $E : 0 \to M \to V^0 \to V^1 \to V^2 \to 0$ and $E' : 0 \to M' \to V^0 \to V^1 \to V^2 \to 0$.
\end{corollary}

\begin{proof}
  The leftward implication follows from additivity in $K_0$ applied to $E$ and $E'$.  To prove the rightward implication, suppose $[M]=[M']$.
  Adding length $1$ acyclic complexes formed from the maps $1_M$ and $1_{M'}$ to the short exact sequences provided by the Lemma yields the
  following exact sequences, of the form desired.
  \[
  \xymatrixrowsep {7pt}
  \xymatrix{
    0 \ar[r] & M'\ar[r] & M' \oplus M \oplus V^0 \ar[r] & V^1 \ar[r] & V^2 \ar[r] & 0 \\
    0 \ar[r] & M \ar[r] & M  \oplus M'\oplus V^0 \ar[r] & V^1 \ar[r] & V^2 \ar[r] & 0 \\
    }\]
\end{proof}

\section{Proof of exactness, part 1}

\begin{lemma}\label{part1}
  The subsequence $ K_0\Om[\cM \xrightarrow F \cN] \to K_0\Om[\cM \to 0] \xrightarrow F K_0\Om[\cN \to 0] $ of the sequence in Theorem \ref{exseqK00000} is exact.
\end{lemma}

\begin{proof}
  By Corollary \ref{seqiscomplex} we know the sequence is a complex.  The identifications in Lemma \ref{relKspecial} allow us to write
  the sequence in the form $ K_0\Om[F] \to K_0\cM \xrightarrow F K_0\cN $, so we consider an arbitrary element $[M]-[M']$ of $K_0\cM$ killed by
  $F$ and try to show it is in the image of the previous map, which we'll call {\em projection}.  By \ref{K0lemmacoro} applied to $FM$ and
  $FM'$ we find exact sequences $0 \to FM \xrightarrow{u} N^0 \xrightarrow{d} N^1 \xrightarrow{d} N^2 \to 0$ and $ 0 \to FM' \xrightarrow{u'}
  N^0 \xrightarrow{d'} N^1 \xrightarrow{d'} N^2 \to 0$ in $\cN$.  Let $N$ denote the binary complex $(N,d,d')$ and regard $u$ and $u'$ as
  quasi-isomorphisms $u : FM \xrightarrow \sim \Top N$ and $u' : FM' \xrightarrow \sim \Bot N$.  (Here we write the same name for an object and
  the corresponding chain complex concentrated in degree $0$.)  The class in $K_0\Om[F]$ of the object $((M,M'), N, (u,u')) \in B[F]$ projects
  to $[M]-[M']$ in $K_0 \cM$, yielding the result.
\end{proof}

\section{Proof of exactness, part 2}

Let $\cM \xrightarrow F \cN$ be an exact functor between exact categories.  For $i \in \NN$ and $(M,N,u) \in B[F]$, let $(M,N,u)[i] \defeq
(M[i],N[i],u[i])$, where the shift operation on chain complexes (or on binary chain complexes) is the standard one, as defined in
\cite[Definition 2.1]{grayson-MR2947948}.

\begin{lemma}\label{K0omegashift}
  For $i \in \NN$ and $(M,N,u) \in B[F]$, the equation $[(M,N,u)[i]] = (-1)^i [(M,N,u)]$ holds in $K_0 \Omega[F]$.
\end{lemma}

\begin{proof}
  We may assume $i = -1$.  The naive filtration of $M$ and $N$ does not necessarily induce a filtration of $u$ in $B[F]$, because the components
  $u_j$ may not be isomorphisms.  Nevertheless, it does induce a filtration of the mapping cone $C$ of the identity map $1_{(M,N,u)}$, because
  identity maps have acyclic mapping cones, and any map between such cones is a quasi-isomorphism.  Moreover, the successive subquotients of the
  filtration are diagonal, hence $[C] = 0$.  Additivity over the cone exact sequence $ 0 \to (M,N,u) \to C \to (M,N,u)[-1] \to 0$ gives the
  result.
\end{proof}

\begin{corollary}
  Any element of $K_0 \Omega[F]$ can be written in the form $[(M,N,u)]$ for some object $(M,N,u)$ of $B[F]$.
\end{corollary}

\begin{proof}
  Using direct sum, an arbitrary element can be written as a difference $[(M,N,u)]-[(M',N',u')]$ of generators, which can be rewritten as
  $[(M\oplus M'[1], N \oplus N'[1], u \oplus u'[1])]$.
\end{proof}

\begin{lemma}\label{part2}
  The subsequence $ K_0\Om[0 \to \cN] \to K_0\Om[\cM \xrightarrow F \cN] \to K_0\Om[\cM \to 0] $ of the sequence in Theorem \ref{exseqK00000} is exact.
\end{lemma}

\begin{proof}
  By Corollary \ref{seqiscomplex} we know the sequence is a complex.  Using the corollary, consider an arbitrary element of the form
  $[(M,N,u)]$ in $K_0\Om[\cM \xrightarrow F \cN]$ that projects to $0$, where $(M,N,u) \in B[F]$, and try to show it is in the image of the
  previous map.  Identifying $K_0\Om[\cM \to 0]$ with $K_0 \cM$ by Lemma \ref{relKspecial}, we see that $\chi (\Top M) = \chi( \Bot M) \in
  K_0\cM$, where $\chi(C) \defeq \sum_i (-1)^i [C_i]$ for a chain complex $C$.  Using that equality, we refer to the (elementary) proof of
  \cite[Lemma 5.4]{grayson-MR2947948}, which shows something slightly stronger than is written in the statement of the lemma: that there exists
  a pair $X \in C^2\cM$ of acyclic chain complexes such that $\gr \Top M \oplus \gr \Top X \isom \gr \Bot M \oplus \gr \Bot X$.  Acyclicity of
  $X$ tells us that $(X,0,0)$ is in $B[F]$ and that the map $0 \to (X,0,0)$ is in $p$, hence that $[(X,0,0)] = 0$ in $K_0 \Omega[F]$.  We
  replace $(M,N,u)$ by $(M,N,u) \oplus (X,0,0)$ without changing its class in $K_0 \Omega[F]$, arriving at a situation where $\gr \Top M \isom
  \gr \Bot M$.  Replacing the objects of $\Bot M$ by the objects of $\Top M$ and composing the differentials of $\Bot M$ with the appropriate
  isomorphisms, we achieve $\gr \Top M = \gr \Bot M$, without changing the isomorphism class of $\Bot M$.  Hence $M = \topbot L$ for some $L \in
  B\cM$.  The quasi-isomorphism $u : \topbot FL \xrightarrow\sim \topbot N$ has an acyclic mapping cone $\Cone(u) \in C^2\cN$.  Since $\gr \Top
  \Cone(u) = \gr N \oplus \gr F L [-1] = \gr \Bot \Cone(u)$, there is a unique object $\Cone'(u) \in B^q\cN$ with $\topbot \Cone'(u) =
  \Cone(u)$.  Thus we get a class $[(0,\Cone'(u),0)] \in K_0\Om[0 \to \cN]$, the first group of our sequence, and the goal is to show it maps to
  our generator.

  The remark in the previous paragraph about mapping cones applies more generally. Given $X,Y\in B[F]$ and a map $f : \topbot X \to \topbot Y$
  in $C[F^2]$, the mapping cone $\Cone(f)$ is uniquely of the form $\topbot \Cone'(f)$ for some $\Cone'(f) \in B[F]$.

  Let $C'[F]$ be defined just as $C[F]$ was, except that the comparison maps inside each object are not required to be quasi-isomorphisms.  Now
  consider the following exact sequences of pairs in $C'[F^2]$, where each pair is displayed as a vertical arrow.
  \[
  \xymatrix{
    0 \ar[r] & 0 \ar[d]\ar[r] &  (M,\topbot FL,1) \ar[d]^{(1,u)} \ar[r]^1 & (M,\topbot FL,1) \ar[d] \ar[r] & 0 \\
    0 \ar[r] & (M,\topbot N,u) \ar[r]^1 &  (M,\topbot N,u) \ar[r] & 0 \ar[r] & 0 \\
    0 \ar[r] & (0,\topbot FL,0) \ar[d]^u \ar[r] & (M,\topbot FL,1) \ar[d]^{(1,u)} \ar[r] & (M,0,0) \ar[d]^1 \ar[r] & 0 \\
    0 \ar[r] & (0,\topbot N,0)  \ar[r] & (M,\topbot N,u) \ar[r] & (M,0,0) \ar[r] & 0
  }
  \]

  Applying $\Cone$ to each of the pairs above yields the following pair of exact sequences in $C'[F^2]$.
  \[\xymatrixrowsep{7pt}
  \xymatrix{
    0 \ar[r] & (M,\topbot N,u) \ar[r] & Y' \ar[r] & W' \ar[r] & 0 \\
    0 \ar[r] & (0,\Cone(u),0) \ar[r] & Y' \ar[r] & Z' \ar[r] & 0
    }
  \]
  The comparison maps in each of the objects here are quasi-isomorphisms, because in every case, either we have the mapping cone of a
  quasi-isomorphism (in 4 of the 6 cases), or we have the mapping cone of a map between two objects each with a comparison map that is a
  quasi-isomorphism (in 4 of the 6 cases).  Thus the exact sequences lie in $C[F^2]$.

  We see that our exact sequences arise from exact sequences in $B[F]$, in which
  $Y \defeq \Cone'(1,u)$.
  \[\xymatrixrowsep{7pt}
  \xymatrix{
    0 \ar[r] & (M,N,u) \ar[r] & Y \ar[r] & W \defeq (M,FL,1)[-1] \ar[r] & 0 \\
    0 \ar[r] & (0,\Cone'(u),0) \ar[r] & Y \ar[r] & Z \defeq (\Cone(1_M),0,0) \ar[r] & 0
    }
  \]

  The naive filtration of $L$ and $M$ induces filtrations of $W$ and $Z$ with diagonal successive subuotients, so $[W]$ and $[Z]$ vanish in $K_0
  \Omega[F]$.  From additivity and the two exact sequences we deduce that $[(M,N,u)] = [Y] = [(0,\Cone' u,0)]$, yielding the result.
\end{proof}

\section{Facts about the Grothendieck group, part 2}

\begin{lemma}\label{K0lemmaimproved}
  Let $\Delta : v\cD \to w \cM$ be an exact functor between exact categories with weak equivalences, in the sense of \cite[Definition
    A.1]{grayson-MR2947948}, and with a cylinder functor in the sense of \cite[1.6]{waldhausen-MR802796}.  Assume that the cone of any
  isomorphism in $\cM$ has an admissible filtration whose successive subquotients are in the image of $\Delta$ (up to isomorphism).  Consider
  objects $M,M' \in \cM$.  Their classes $[M]$ and $[M']$ in the relative Grothendieck group $K_0[\Delta]$ of the pair $[\Delta]$ are equal
  if and only if there exist objects $V^0, V^1, V^2$ of $\cM$, mapping cones $C$ and $C'$ of maps in $w$, objects $D$ and $D'$ in the
  image of $\Delta$, and exact sequences of the form $E : 0 \to M \oplus C \oplus D \oplus V^0 \to V^1 \to V^2 \to 0$ and $E' : 0 \to M'
  \oplus C' \oplus D' \oplus V^0 \to V^1 \to V^2 \to 0$.
\end{lemma}

\begin{remark}
  A cylinder functor supports also the construction of cones and suspensions.  If $w\cM$ is the category of chain complexes over an exact
  category, then the mapping cylinder for chain complexes provides a cylinder functor, and if $w$ contains the quasi-isomorphisms, then the
  cylinder axiom holds; the cone is defined as the cokernel of the front inclusion into the cylinder and agrees with the mapping cone; and the
  suspension of $M$ is defined as the cone of $M \to 0$ and agrees with the functor $M \mapsto M[-1]$. The assumption about filtrations of
  mapping cones in the lemma is a substitute in our context for the cylinder axiom.
\end{remark}

\begin{remark}\label{relKdefquot}
  The Grothendieck group $K_0[\Delta]$ is the quotient of $K_0\cM$ by the subgroup $H$ generated by all elements of the form $[F(D)]$ for some
  $D \in \cD$ together with all elements of the form $[M] - [M']$ for some $M' \to M$ in $w$.  It is independent of $v$.  For the elementary
  purposes of this paper, that may be taken as the definition.
\end{remark}

\begin{proof}
  Given $f : M' \to M$ in $\cM$, consider the following pair of exact sequences.
  \[\xymatrixrowsep{7pt}
  \xymatrix {
    0 \ar[r] &  M \ar[r] &  \Cone f \ar[r] & M'[-1] \ar[r] & 0 \\
    0 \ar[r] &  M' \ar[r] &  \Cone 1_{M'} \ar[r] & M'[-1] \ar[r] & 0
    }
  \]
  By the hypothesis on filtrations, we know $[\Cone 1_{M'}]$ vanishes in $K_0[\Delta]$, so by additivity, $[\Cone f] = [M] - [M']$.  In
  particular, cones of maps in $w$ vanish in $K_0[\Delta]$, establishing the leftward implication.

  Write $K_0[\Delta]$ as the quotient $(K_0\cM)/H$, where $H$ is the subgroup described by \ref{relKdefquot}, and use the notation $[M]_0$
  temporarily to denote the class of $M$ in $K_0 \cM$, thereby distinguishing it from $[M] \in K_0[\Delta]$.  

  A further consequence of the hypothesis on filtrations and the argument above is that we may regard $H$ as being generated by all elements of
  the form $[F(D)]_0$ for some $D \in \cD$ together with all elements of the form $[\Cone f]_0$ for some $f \in w$.

  To prove the rightward implication, consider objects $M,M' \in \cM$ and suppose $[M]=[M']$.  It follows that $[M]_0-[M']_0 \in H$ and can be
  written as a linear combination of generators.  Using direct sums to combine like terms with coefficients of the same sign in the linear
  combination, we may write $[M]_0-[M']_0 = [D']_0 - [D] + [C']_0 - [C]_0$ for some objects $D$ and $D'$ in the image of $\Delta$ and for some
  mapping cones $C$ and $C'$ of maps in $w$.  Thus $[M \oplus C \oplus D]_0 = [M' \oplus C' \oplus D']_0$, and an application of
  \ref{K0lemma} gives the result.
\end{proof}

\section{Facts about \texorpdfstring{$K_1$}{K sub 1}}

For forming the total complex of a complex $N$ of complexes, or for forming the total binary complex of a binary complex of complexes, and so
on, we will fix the following convention for adjusting the signs of the differentials to produce the differential of the total complex: we alter
the sign of the differential of the complex $N_i$ by $(-1)^i$.  That is equivalent to rewriting the chain complex $ \dots \to N_i \xrightarrow d
N_{i-1} \to \dots$ where each $d$ is a chain map of degree $0$ as a chain complex $ \dots \to N_i[-i] \xrightarrow d N_{i-1}[-i+1] \to \dots$
where each $d$ is a chain map of degree $-1$ obeying the usual sign rule.  Observe that any chain complex $(M,d)$ is isomorphic to $(M,-d)$, and
isomorphic objects have the same class in Grothendieck groups, so the change of sign causes no problem in practice.

\begin{lemma}\label{K1shift}
  Given $N \in B^q\cN$, the equation $[N[i]] = (-1)^i [N]$ holds in $K_1\cN$.
\end{lemma}

\begin{proof}
  Apply \ref{K0omegashift} to the exact functor $0 \to \cN$.
\end{proof}

An acyclic binary chain complex $ \xymatrix {N_1 \ar[r]<2pt> \ar[r]<-2pt> & N_0 } $ of length 1 is essentially the same
thing as an automorphism $N_0 \xrightarrow\isom N_0$.  We formalize that as follows.

\begin{definition}\label{K1autodef}
  If $\theta$ is an automorphism of an object $N$ of an exact category $\cN$, let $[\theta] \defeq [\cA(\theta)] \in K_1\cN$, where $\cA(\theta)$
  is the acyclic binary complex of $\cN$ with copies of $N$ in degrees $0$ and $1$, with top differential $\theta : N \to N$, and with bottom
  differential $1 : N \to N$.
\end{definition}

Any acyclic binary chain complex of $\cN$ of the form $\xymatrix{N_1 \ar[r] <2pt> ^{f} \ar[r] <-2pt> _{g} & N_0}$ is isomorphic to $ \cA( f \circ
g^{-1} )$.

\begin{definition}\label{elemauto}
  An automorphism of an object $N$ of an exact category $\cN$ is called {\em elementary} if it has the form $\begin{psmallmatrix} 1 & * \\ 0 &
    1 \end{psmallmatrix}$ with respect to some direct sum decomposition $N \isom N' \oplus N''$ of $N$.
\end{definition}

\begin{lemma}\label{K1autoadd}
  Let $\cN$ be an exact category, and let $N$ be an object of $\cN$.
  \begin{enumerate}
    \item
      If $\theta$ and $\psi$ are automorphisms of $N$, then $[\theta \psi] = [\theta] + [\psi] \in K_1 \cN$.
    \item
      If $\theta$ is a product of elementary automorphisms, then $[\theta] = 0$.
  \end{enumerate}
\end{lemma}

\begin{proof}
  For (1), consider the following pair of bicomplexes in $\cN$ with the same underlying $\ZZ^2$-graded object, concentrated in degrees $\{0,1\}
    \times \{0,1\}$.
  \begin{mathparpagebreakable}
    \xymatrix{ N \ar[r]^{-1}\ar[d]^\psi & N \ar[d]^{\theta \psi} \\ N \ar[r]^\theta & N } \and
    \xymatrix{ N \ar[r]^{-1}\ar[d]^1 & N \ar[d]^{1} \\ N \ar[r]^1 & N } 
  \end{mathparpagebreakable}
  (The diagrams above anticommute, as they ought to in a bicomplex.)
  The total complexes give a single object $K$ of $B^q\cN$.  Filter by rows (one of which is diagonal) to see that $[K] = [\theta] \in K_1 \cN$.
  Filter by columns and apply \ref{K1shift} to see that $[K] = [\theta \psi ] - [\psi]$, yielding the result.

  For (2), by (1) we may assume $\theta$ is elementary with respect to a direct sum decomposition $N \isom N' \oplus N''$ of $N$.  Additivity
  over the short exact sequence $ 0 \to \cA ( 1_{N'} ) \to \cA ( \theta ) \to \cA ( 1_{N''} ) \to 0 $ gives the result.
\end{proof}

\begin{lemma}\label{K1topbotiso}
  Suppose $M$ and $N$ are objects of $B^q\cN$, and $f : \Top M \xrightarrow \isom \Top N$ and $g : \Bot M \xrightarrow \isom \Bot N$ are
  isomorphisms of complexes.  Then the following equation holds in $K_1 \cN$.
  \[
    [N]-[M] = \sum_i (-1)^i [ f_i g_i^{-1}]
  \]
\end{lemma}

\begin{proof}
  Form two complexes of complexes of $\cN$ with the same underlying $\ZZ^2$-graded object that have $M$ in
  column 1 and $N$ in column 0, using $f$ and the top differentials of $M$ and $N$ in the first and using $g$ and the bottom differentials in
  the second.  All rows and columns are acyclic, so the total complexes form an object $K$ of $B^q\cN$.  Filter by columns and apply
  \ref{K1shift} to see that $[K] = [N] - [M]$.  Filter by rows to get acyclic binary complexes of length 1 of the form $\xymatrix{M_i \ar[r]
    <2pt> ^{f_i} \ar[r] <-2pt> _{g_i} & N_i}$ supported in degrees $\{i, i+1\}$, which are in turn isomorphic to acyclic binary complexes of the
  form $\cA(f_i g_i^{-1})[-i]$.
\end{proof}

\begin{lemma}\label{K1diffrot}
  Suppose that $d_1, \dots, d_n$ are acyclic differentials on a graded object $N$ of $\cN$, and let $\sigma$ be a permutation of the set
  $\{1,\dots,n\}$.  Then the following equation holds in $K_1\cN$.
  \[
    [(N,d_1,d_{\sigma 1})] + [(N,d_2,d_{\sigma 2})] + \dots + [(N,d_n,d_{\sigma n})] = 0
    \]
\end{lemma}

\begin{proof}
  A permutation can be written as a product of disjoint cycles, so by treating each cycle separately we may assume $\sigma$ is a cycle of order
  $n$.  By renumbering, we may assume $\sigma 1 = 2$, $\sigma 2 = 3$, \dots, $\sigma n = 1$.

  The signed block permutation matrix 
  \[
  S \defeq 
  \begin{psmallmatrix}
    0 & -1 & \dots & 0 & 0 \\
    0 & 0  & \dots & 0 & 0 \\
    \vdots & \vdots  & \ddots & \vdots & \vdots \\
    0 & 0  & \dots & 0 & -1 \\
    1 & 0  & \dots & 0 & 0
  \end{psmallmatrix}
  \]
  gives an automorphism of $N^n$ that is a product of elementary matrices.  (For example, when $n=2$, one uses the following equation.
  \[
  \begin{pmatrix} 0 & -1 \\ 1 & 0 \end{pmatrix}
  =
  \begin{pmatrix} 1&0\\1&1 \end{pmatrix}
  \begin{pmatrix} 1&-1\\0&1 \end{pmatrix}
  \begin{pmatrix} 1&0\\1&1 \end{pmatrix}
  \]
  The general case follows from the case for $n=2$ by writing $S$ as a product of signed transpositions.)
  Moreover, $S$ satisfies the equation 
  \[
  S ( d_1 \oplus \dots \oplus d_n ) = 
  \begin{psmallmatrix}
    0 & -d_2 & \dots & 0 & 0 \\
    0 & 0  & \dots & 0 & 0 \\
    \vdots & \vdots  & \ddots & \vdots & \vdots \\
    0 & 0  & \dots & 0 & -d_n \\
    d_1 & 0  & \dots & 0 & 0
  \end{psmallmatrix}
  = ( d_2 \oplus \dots \oplus d_n \oplus d_1 ) S
  \]
  and thus gives an isomorphism $S : (N, d_1 \oplus \dots \oplus d_n) \xrightarrow \isom (N, d_2 \oplus \dots \oplus d_n \oplus d_1) $ in $C^q
  \cN$.  Now apply \ref{K1topbotiso} to the objects $(N^n,d_1 \oplus \dots \oplus d_n,d_1 \oplus \dots \oplus d_n)$
  and $(N^n,d_1 \oplus \dots \oplus d_n,d_2 \oplus \dots \oplus d_n \oplus d_1)$ along with \ref{K1autoadd} to get 
  $
  [(N,d_1,d_2)] + [(N,d_2,d_3)] + \dots + [(N,d_{n-1},d_n)] + [(N,d_n,d_1)] = [(N^n,d_1 \oplus \dots \oplus d_n,d_2 \oplus \dots \oplus d_n
    \oplus d_1)] = [(N^n,d_1 \oplus \dots \oplus d_n,d_1 \oplus \dots \oplus d_n)] = 0.$
\end{proof}

\begin{corollary}\label{K1sw}
  If $N \in B^q \cN$, then $[\sw N] = -[N]$ in $K_1\cN$.
\end{corollary}

\begin{proof}
  This follows from the case $n=2$ of the lemma.
\end{proof}

\section{Proof of exactness, part 3}

Let $\cM \xrightarrow F \cN$ be an exact functor between exact categories.

\begin{lemma}\label{TotBCF}
  Given $E \in BC[F]$ such that $\Tot E \in B[F]$ is acyclic, the classes $[\Tot E_\src] \in K_1 \cM$ and $[\Tot E_\tar] \in K_1 \cN$ are
  related by the equation $F [\Tot E_\src] = [\Tot E_\tar]$.
\end{lemma}

\begin{proof}
  By the exact sequence of a mapping cone, it suffices to show that $\beta \defeq [\Cone ( \Tot E_\qi : \Tot FE_\src \to \Tot E_\tar )] \in K_1
  \cN$ vanishes.  Observe that for each $i \in \ZZ$, the map $E^i_\qi : FE^i_\src \to E^i_\tar$ is a quasi-isomorphism in $C \cN$, and hence its
  mapping cone is acyclic.  Thus the successive subquotients of the naive filtration of $E$ (in the $B$-direction) are in $B^q \cN$; moreover,
  they are diagonal, because the two differentials in the $B$-direction no longer contribute and there was only one differential in the
  $C$-direction.  Hence their classes vanish in $K_1 \cN$, and thus so does their sum $\beta$, as desired.
\end{proof}

\begin{lemma}\label{part3}
  The subsequence $ K_0\Om[0 \to \cM] \xrightarrow F K_0\Om[0 \to \cN] \to K_0\Om[\cM \xrightarrow F \cN]$ of the sequence in Theorem \ref{exseqK00000} is exact.
\end{lemma}

\begin{proof}
  By Corollary \ref{seqiscomplex} we know the sequence is a complex.   Consider an arbitrary element $[N] - [N']$ of $K_0\Om[0 \to \cN]$
  that dies in $K_0\Om[\cM \xrightarrow F \cN]$.  By \ref{K0lemmaimproved} we can find a pair of exact sequences
  \[ 
  \xymatrixrowsep {7pt}
  \xymatrix {
    Q : & 0 \ar[r] & N \oplus C \oplus D \oplus V^0 \ar[r] & V^1 \ar[r] & V^2 \ar[r] & 0 \\
    Q' : & 0 \ar[r] & N' \oplus C' \oplus D' \oplus V^0 \ar[r] & V^1 \ar[r] & V^2 \ar[r] & 0 
    }
  \]
  in $B[F]$, in which $C$ and $C'$ are mapping cones of maps in $p$ and $D$ and $D'$ are diagonal.  Since $N$, $N'$, $C$, and $C'$ are acyclic,
  the subcomplexes 
  \[ 
  \xymatrixrowsep {7pt}
  \xymatrixcolsep {44pt}
  \xymatrix {
    L :  & D \oplus V^0 \ar[r]^-{(p \,\, e)} & V^1 \ar[r]^e & V^2  \\
    L' : & D' \oplus V^0 \ar[r]^-{(p' \, e')} & V^1 \ar[r]^{e'} & V^2  
    }
  \]
  have acyclic total complexes in $B[F]$.  Applying the functors $\Top$ and $\Bot$ yields the complexes $\Top L$, $\Top L'$, $\Bot L$, and $\Bot
  L'$ of objects in $C[F]$ .  Up to a permutation of direct summands, the objects of $C[F]$ involved in the two complexes $\Top L \oplus \Bot
  L'$ and $\Top L' \oplus \Bot L$ are equal.  To achieve equality, we implement the permutation of direct summands explicitly to get the
  following pair of complexes $A$ and $A'$ connecting the same objects in $C[F]$, each with acyclic total complexes, and with isomorphisms $\theta : A \isom
  \Top L \oplus \Bot L'$ and $\theta' : A' \isom \Top L' \oplus \Bot L$, where the isomorphisms are implemented as permutations of direct
  summands, possibly with signs.
  \[ 
  \xymatrixrowsep {12pt}
  \xymatrix {
    A :  & \Top D \oplus \Bot D' \oplus \Top V^0 \oplus \Bot V^0 \ar[r]^-{q} \ar@2{=}[d] & \Top V^1 \oplus \Bot V^1 \ar[r]^-{r} \ar@2{=}[d] & \Top V^2 \oplus \Bot V^2 \ar@2{=}[d] \\
    A':  & \Bot D \oplus \Top D' \oplus \Top V^0 \oplus \Bot V^0 \ar[r]^-{q'} & \Top V^1 \oplus \Bot V^1 \ar[r]^-{r'} & \Top V^2 \oplus \Bot V^2 
    }
  \]
  (The diagram above does not commute.)
  The matrices of the four maps are given by the following equations.
  \[
    \begin{aligned}
      q \defeq & \begin{pmatrix} \Top p & 0 & \Top e & 0 \\ 0 & \Bot p' & 0 & \Bot e' \end{pmatrix} &\quad& r \defeq & \begin{pmatrix} \Top e & 0 \\ 0 & \Bot e' \end{pmatrix} \\
      q'\defeq & \begin{pmatrix} 0 & - \Top p' & \Top e' & 0 \\ \Bot p & 0 & 0 & \Bot e \end{pmatrix} && r'\defeq & \begin{pmatrix} \Top e' & 0 \\ 0 & \Bot e  \end{pmatrix}
    \end{aligned}
    \]
  We have inserted an unmotivated minus sign so the computation below comes out right, compensating for it with a minus sign in ${\theta'}$.

  Let $E$ be the object of $BC[F]$ provided by the pair $(A,A')$; its total complex is an acyclic object of $B[F]$.  Applying Lemma \ref{TotBCF}
  to $E$, we see that it suffices to show that $[N]-[N'] = [\Tot E_\tar]$ in $K_1 \cN$.

  The exact sequences 
  \[ 
  \xymatrixrowsep {7pt}
  \xymatrix {
    0 \ar[r] & \Tot L \ar[r] & \Tot Q \ar[r] & N \oplus C \ar[r] & 0 \\
    0 \ar[r] & \Tot L' \ar[r] & \Tot Q' \ar[r] & N' \oplus C' \ar[r] & 0
    }\]
  allow us to deduce that $[N] - [N'] = [\Tot L_\tar] - [\Tot L'_\tar] + [C'_\tar] - [C_\tar]$ in $K_1 \cN$.  

  Since $C_\tar$ and $C'_\tar$ are mapping cones of isomorphisms of binary complexes, the naive filtration of those binary complexes induces a
  filtration on them whose successive quotients are diagonal and acyclic, showing that $[C'_\tar] = [C_\tar] = 0$.  Hence it is enough to show
  that $[\Tot L_\tar] - [\Tot L'_\tar] = [\Tot E_\tar]$ in $K_1 \cN$.

  Notice the difference between two meanings of $\Tot$ here: $\Tot L_\tar$ and $\Tot L'_\tar$ are total binary complexes of complexes of binary complexes,
  whereas $\Tot E_\tar$ is the total binary complex of a binary complex of complexes.

  We implement the direct sum $L \oplus \sw L'$ up to a permutation $B \isom L \oplus \sw L'$ of summands by defining $B \in CB[F]$ to be the following
  chain complex of objects of $B[F]$.
  \[
  D \oplus \sw D' \oplus V^0 \oplus \sw V^0 
  \xrightarrow {\begin{pmatrix} p&0&e&0 \\ 0&\sw p'&0&\sw e' \end{pmatrix}} 
  V^1 \oplus \sw V^1 
  \xrightarrow {\begin{pmatrix} e&0 \\ 0&\sw e' \end{pmatrix}} 
  V^2 \oplus \sw V^2 
  \]

  Applying the functors $\Top$ and $\Bot$ yields the following chain complexes, displayed here as $A$ and $A'$ were displayed above.
  \[ 
  \xymatrixrowsep {12pt}
  \xymatrix {
    \Top B : & \Top D \oplus \Bot D' \oplus \Top V^0 \oplus \Bot V^0 \ar[r]^-{q}   & \Top V^1 \oplus \Bot V^1 \ar[r]^-{r}   & \Top V^2 \oplus \Bot V^2 \\
    \Bot B : & \Bot D \oplus \Top D' \oplus \Bot V^0 \oplus \Top V^0 \ar[r]^-{q''} & \Bot V^1 \oplus \Top V^1 \ar[r]^-{r''} & \Bot V^2 \oplus \Top V^2 
    }
  \]
  Here $q$ and $r$ are as in $A$ above, so $\Top B = A = \Top E$, but $q''$ and $r''$ differ from $q'$ and $r'$ as follows.
  \[
    \begin{aligned}
      q''\defeq & \begin{pmatrix} \Bot p & 0 & \Bot e & 0 \\ 0 & \Top p' & 0 & \Top e' \end{pmatrix} &\quad\quad& r''\defeq & \begin{pmatrix} \Bot e & 0 \\ 0 & \Top e' \end{pmatrix}
    \end{aligned}
    \]

  It suffices to show that $[\Tot E_\tar] = [\Tot B_\tar]$.  We know that $\Top \Tot B_\tar = \Top \Tot E_\tar$, so we introduce a permutation isomorphism
  $g : \Bot E \xrightarrow\isom \Bot B$ defined as follows, intending to apply \ref{K1topbotiso}.
  \[
  \begin{aligned}
    g^0 = \begin{pmatrix} 1&0&0&0 \\ 0&1&0&0 \\ 0&0&0&1 \\ 0&0&-1&0 \end{pmatrix} & : \Bot D \oplus \Top D' \oplus \Top V^0 \oplus \Bot V^0 \xrightarrow \isom \Bot D \oplus \Top D' \oplus \Bot V^0 \oplus \Top V^0 \\
    g^1 = \begin{pmatrix} 0&1 \\ -1&0 \end{pmatrix} & : \Top V^1 \oplus \Bot V^1 \xrightarrow \isom \Bot V^1 \oplus \Top V^1 \\
    g^2 = \begin{pmatrix} 0&1 \\ -1&0 \end{pmatrix} & : \Top V^2 \oplus \Bot V^2 \xrightarrow \isom \Bot V^2 \oplus \Top V^2 \\
  \end{aligned}
  \]

  The maps $g^i$ are chain maps, so it suffices to check that the following diagram commutes.
  \[ 
  \xymatrix {
    \Bot E \ar[d]^g : & \Bot D \oplus \Top D' \oplus \Top V^0 \oplus \Bot V^0 \ar[r]^-{q'} \ar[d]^-{g^0} & \Top V^1 \oplus \Bot V^1 \ar[r]^-{r'} \ar[d]^-{g^1} & \Top V^2 \oplus \Bot V^2 \ar[d]^-{g^2}  \\
    \Bot B : & \Bot D \oplus \Top D' \oplus \Bot V^0 \oplus \Top V^0 \ar[r]^-{q''} & \Bot V^1 \oplus \Top V^1 \ar[r]^-{r''} & \Bot V^2 \oplus \Top V^2 
    }
  \]
  Commutativity of the left hand square amounts to the following computation.
  \[
  \begin{aligned}
    g^1 q' & = 
    \begin{pmatrix} 0&1 \\ -1&0 \end{pmatrix} \begin{pmatrix} 0 & - \Top p' & \Top e' & 0 \\ \Bot p & 0 & 0 & \Bot e \end{pmatrix} \\
    & = \begin{pmatrix} \Bot p & 0 & 0 & \Bot e \\ 0 & \Top p' & - \Top e' & 0 \end{pmatrix} \\
    & = \begin{pmatrix} \Bot p & 0 & \Bot e & 0 \\ 0 & \Top p' & 0 & \Top e' \end{pmatrix} \begin{pmatrix} 1&0&0&0 \\ 0&1&0&0 \\ 0&0&0&1 \\ 0&0&-1&0 \end{pmatrix} \\ 
    & = q'' g^0
  \end{aligned}
  \]
  Commutativity of the right hand square is the same, except for the absence of the contributions from $D$ and $D'$.

  The map $g$ induces an isomorphism $\Tot g_\tar : \Tot \Bot E_\tar \xrightarrow\isom \Tot \Bot B_\tar$ that serves as a companion for the
  identity $\Tot \Top E_\tar = \Tot \Top B_\tar$.  Since it is a product of elementary matrices in each degree, \ref{K1autoadd} and
  \ref{K1topbotiso} yield the desired equality, $[\Tot E_\tar] = [\Tot B_\tar]$.
\end{proof}

\section{Proof of exactness, conclusion}

\begin{proof}[Proof of \ref{exseqK00000}]
  Apply \ref{part1}, \ref{part2}, and \ref{part3}.
\end{proof}

\bibliographystyle{hamsplain}
\bibliography{papers}

\providecommand{\bysame}{\leavevmode\hbox to3em{\hrulefill}\thinspace}
\begin{thebibliography}{1}

\bibitem{grayson-MR2947948}
Daniel~R. Grayson, \emph{Algebraic {$K$}-theory via binary complexes}, J. Amer.
  Math. Soc. \textbf{25} (2012), no.~4, 1149--1167, links:
  \href{http://www.math.uiuc.edu/K-theory/0988/}{preprint}.

\bibitem{waldhausen-MR802796}
Friedhelm Waldhausen, \emph{Algebraic {$K$}-theory of spaces}, Algebraic and
  geometric topology ({N}ew {B}runswick, {N}.{J}., 1983), Lecture Notes in
  Math., vol. 1126, Springer, Berlin, 1985, pp.~318--419, links:
  \href{http://www.math.uni-bielefeld.de/~fw/}{homepage}
  \href{http://www.math.uni-bielefeld.de/~fw/LNM1126_318-419.pdf}{pdf}
  \href{http://www.math.uni-bielefeld.de/~fw/LNM1126_318-419.ocr.djvu}{djvu}
  \href{http://www.math.uni-bielefeld.de/~fw/algebraic_K_theory_of_spaces.pdf}{pdf-amended}
  \href{http://www.math.uni-bielefeld.de/~fw/algebraic_K_theory_of_spaces.djvu}{djvu-amended}.

\end{thebibliography}


\end{document}